\documentclass[12pt,reqno]{amsproc}

\usepackage[margin = 1in]{geometry}

\usepackage{comment}
\usepackage[small,bf]{caption}
\usepackage{nicefrac}
\usepackage[usenames, dvipsnames]{color}
\usepackage{hyperref}
\hypersetup
{
    colorlinks,	%
    citecolor=blue,%
    filecolor=black,%
    linkcolor=blue,%
    urlcolor=blue
}
\usepackage{amsmath}
\usepackage{amsthm}
\usepackage{amssymb}
\usepackage{amsfonts}
\usepackage{mathabx}
\usepackage{eulervm, xfrac}
\usepackage[sans]{dsfont}
\usepackage{mathdots}
\usepackage{mathrsfs}
\usepackage{rotating, setspace}
\DeclareMathAlphabet{\mathbfsf}{\encodingdefault}{\sfdefault}{bx}{n}
\usepackage[all]{xy}
\usepackage{tikz-cd}
\usepackage{enumerate}
\usepackage[capitalise]{cleveref}

\usepackage{thmtools,thm-restate}

\DeclareFontShape{OT1}{cmr}{bx}{sc}{<-> cmbcsc10}{}

\usepackage[colorinlistoftodos]{todonotes}

\newcommand{\define}[1]{\textbf{#1}}

\newcommand{\cosheaf}[1]{\widehat{#1}}
\theoremstyle{definition}
\newtheorem{thm}{Theorem}[section]

\newtheorem{defn}[thm]{Definition}

\newtheorem{ex}[thm]{Example}

\newtheorem{rmk}[thm]{Remark}

\newcommand{\calJ}{\mathcal{J}}
\newcommand{\cJ}{\mathcal{J}}

\newcommand{\covU}{\mathcal{U}}

\newcommand{\cV}{\mathcal{V}}

\newcommand{\aat}{\mathbfsf{A}}

\newcommand{\bat}{\mathbfsf{B}}

\newcommand{\cat}{\mathbfsf{C}}

\newcommand{\Open}{\mathbfsf{Open}}
\newcommand{\Down}{\mathbfsf{Down}}

\newcommand{\op}{\mathbfsf{op}}

\newcommand{\Lan}{\mathsf{Lan}}
\newcommand{\Ran}{\mathsf{Ran}}

\newcommand{\slfp}{\mathsf{p}}
\newcommand{\slfj}{\mathsf{j}}

\newcommand{\id}{\text{id}}

\newcommand{\pP}{\mathcal{P}}
\newcommand{\pQ}{\mathcal{Q}}

\newcommand{\topX}{\mathcal{X}}

\usepackage[backend = bibtex, url=false,sorting=nyt,style=numeric-comp,maxbibnames = 99]{biblatex}
\bibliography{erratum}

\begin{document}

\title{Functors on Posets Left Kan Extend to Cosheaves: \\ an \emph{Erratum}}

\author{Justin Michael Curry}
\maketitle

\begin{abstract}
In this note we give a self-contained proof of a fundamental statement in the study of cosheaves over a poset. Specifically, if a functor has domain a poset and co-domain a co-complete category, then the left Kan extension of that functor along the embedding of the domain poset into its poset of down-sets is a cosheaf.
This proof is meant to replace the mistaken proofs published in the author's thesis and an article on dualities exchanging cellular sheaves and cosheaves.
\end{abstract}

\section{Introduction}

The study of sheaves and cosheaves on posets has gained recent renewed attention.
This is in part due to their deployment in more applied settings such as topological data analysis, where they provide useful reformulations of persistent homology~\cite{Curry2014,kashiwara2018persistent,berkouk2019ephemeral} and Reeb graphs~\cite{deSilva2016}. Other examples include network coding and signal processing~\cite{ghrist2011network,Robinson2013,hansen2018toward}.
However, despite this remarkable influx of new activity, the study of sheaves on posets has a long history and similar perceptions of a common idea have appeared again and again.

Sir Eric Christopher Zeeman introduced ``simplicial and \v{C}ech analogues of Leray's sheaf theory'' in his 1955 thesis, which was published in three articles~\cite{zeeman1962dihomology,zeeman1962dihomology2,zeeman1963dihomology}.
A more thorough investigation of sheaves on posets and their connection to Whitney numbers was undertaken by Bac\l{}awski~\cite{baclawski1975}.
Poset-theoretic descriptions of constructible sheaves were provided by Kashiwara~\cite{kashiwara1984riemann} and independently Shepard~\cite{shepard1986cellular}, who wrote their thesis under the direction of MacPherson.
Since then, the utility of studying sheaves and sheaf cohomology groups over posets has continued to be demonstrated, with the works of Yuzvinsky~\cite{yuzvinsky1991cohomology}, Yanagawa~\cite{yanagawa2001sheaves}, and Ladkani~\cite{ladkani2008} serving as some notable waypoints between the past and present.

However, why a functor $F:\pP\to \cat$ defines a sheaf is perhaps easy to explain, but not so easy to carefully prove. 
The explanation proceeds in two steps. First one topologizes $\pP$ by declaring up-sets\footnote{A subset $U\subseteq \pP$ is a \define{up-set} if whenever $p\in U$ and $p\leq q$, then $q\in U$.} to be open and observes that the association of a point $p$ to its principal up-set $U_p:=\{q \mid p\leq q\}$ provides a natural functor $\iota: \pP \to \Open(\pP)^{\op}$.
The second step says that the right Kan extension of $F$ along $\iota$, written $\Ran_{\iota} F$, provides a natural way of assigning data to open sets in a poset, which should be a sheaf.
However, a careful proof that this second step does what is promised is not something that is easy to find in the literature. 
Such a proof should also be easily dualizable to cosheaves.


For some history, in~\cite{Curry2014} such a proof was proposed by using refinement of covers.
This argument is flawed as the counterexample in \cref{sec:refinement-mistake} shows.
This means that Corollary 2.4.4 of~\cite{Curry2014} and Corollary 2.15 of~\cite{curry2018dualities} are incorrect as stated.
This is problematic because Proposition 3.3 and Corollary 3.5 of~\cite{curry2018dualities}, which states that Kan extensions of functors modeled on posets give rise to sheaves and cosheaves on the Alexandrov topology, rests on the validity of these statements. 
However, the truth of those statements is saved by providing a new, self-contained proof that the left Kan extension of a functor along the embedding of a poset into its collection of down-sets is a cosheaf.
The statement for sheaves is easily dualized from here.
It should be noted that the original broken proof, the discovery of the mistake and the currently proposed fix here are all due to the author.

Along the way we introduce two notions of a cosheaf, which rest on two different notions of a cover.
These notions of a cover come from working with a basis of a topological space.
In our setting, principal down-sets serve as a basis for the Alexandrov topology.
This note aims to provide one part of a hopefully growing literature on the subject of foundations for (co)sheaf theory that is adapted to a basis of a space, but where we don't make any concreteness assumptions on the category $\cat$, such as in~\cite{stacks-project}.
Coming up with proofs and theorem statements that naturally dualize between sheaves and cosheaves would provide a welcome clarification at this point in the history of the theory.

\section{Background on Cosheaves}
\label{sec:cosheaves}

First we introduce three notions of a cover.
Let $X$ be a topological space and let $\Open(X)$ be the poset of open sets of $X$, ordered by containment.

\begin{defn}
Fix an open set $U$ and let $\covU=\{U_i\} \subseteq \Open(X)$ be a collection of open sets. 
\begin{enumerate}
	\item We say that $\covU$ is a \define{cover} of $U$ if the union of elements in $\covU$ is $U$.
	\item We say $\covU$ is a \define{\v{C}ech cover} of $U$ if $\covU$ is a cover of $U$ with the additional property that whenever a finite collection of $\{U_i\}_{i \in \sigma}\subset \covU$ has non-empty intersection $U_{\sigma}=\cap_{i\in \sigma} U_i$, then $U_{\sigma}\in \covU$.
	\item Finally, we say $\covU$ is a \define{basic cover} of $U$ if $\covU$ is a cover of $U$ with the property that whenever $U_i, U_j\in \covU$, then $U_i\cap U_j$ is the union of elements in $\covU$.
\end{enumerate}
\end{defn}

\begin{rmk}
The notion of a basic cover comes from considering the defining properties of a basis for a topological space $\topX$. A basis is rarely closed under intersection, but the intersections are unions of elements of the basis.
\end{rmk}

\begin{rmk}
The term \v{C}ech cover is borrowed from Dugger and Isaksen's article~\cite{dugger2004topological}.
The notion of a basic cover is closely related to the notion of a \emph{complete cover} given in the same article. 
The difference is that a basic cover requires that pairwise intersections be covered, whereas a complete cover requires that all finite intersections be covered by elements of the cover.
To see the difference, consider three open sets that have all possible intersections, i.e.~their nerve is a 2-simplex.
A \v{C}ech cover would have all the pair-wise intersections, but without the triple intersection; this corresponds to ``throwing in'' the 1-skeleton of the nerve into the cover. A complete cover would include the triple intersection as well, so the entire nerve would be included in a complete cover.
\end{rmk}

Note that we can regard the inclusion of the cover $\iota_{\covU} : \covU \hookrightarrow \Open(X)$ as an order-preserving map of posets.
Additionally, any poset can be viewed as a category with one object for each element and a unique morphism from one object to another whenever one element is less than another. 
The condition that a map of posets is order-preserving is exactly the condition that says the map of posets defines a functor.
We further note that the condition that $\covU$ is a cover of an open set $U$ is exactly the statement that the colimit of $\iota_{\covU}$ is $U$.
This brings us to two notions of a cosheaf.
We note that the corresponding definitions for sheaves can be easily dualized from here.

\begin{defn}\label{defn:cosheaf}
Let $X$ be a topological space and let $\cat$ be a category with all colimits.
A functor $\cosheaf{F}:\Open(X) \to \cat$ is a \define{cosheaf for $\covU$} if $\covU$ is a \v{C}ech cover and the universal arrow
\[
	\cosheaf{F}[\covU]:=\varinjlim \cosheaf{F} \circ \iota_{\covU} \to F(\varinjlim \iota_{\covU}) = F(U)
\]
is an isomorphism.
Moreover we say that $\cosheaf{F}$ is a \define{cosheaf} if for every \v{C}ech cover $\covU$, the functor $\cosheaf{F}$ is a cosheaf for $\covU$. 

More generally, we say that $\cosheaf{F}$ is a \define{basic cosheaf for $\covU$} if $\covU$ is a \emph{basic} cover and the above universal arrow is an isomorphism.
A \define{basic cosheaf} is then a functor such that for every basic cover the above arrow is an isomorphism. 
\end{defn}

\begin{rmk}
Note that every basic cosheaf is a cosheaf by virtue of the fact that every \v{C}ech cover is a basic cover. 
If we think of a cosheaf as being a functor that commutes with colimits of type $\iota_{\covU}$ when $\covU$ is \v{C}ech cover, then basic cosheaves are functors that commute with a broader class of colimits, namely those functors $\iota_{\covU}$ for which $\covU$ is a basic cover.
\end{rmk}

\section{A Mistaken Argument Using Refinement of Covers}\label{sec:refinement-mistake}

Recall that we say that a cover $\covU_1$ \define{refines} $\covU_2$ if whenever there is a $U'\in \covU_2$, there is an element $U\in \covU_1$ such that $U\subseteq U'$.
Corollary 2.4.4 of~\cite{Curry2014} and Corollary 2.15 of~\cite{curry2018dualities} makes the following, incorrect, assertion:

\begin{quote}
{\em If $\covU_1$ refines $\covU_2$ and if $\cosheaf{F}$ is a cosheaf for $\covU_1$, then it is a cosheaf for $\covU_2$. In other words, if the curved arrow below is an isomorphism, then the other arrows must be isomorphisms as well.}
\end{quote}

\[
\xymatrix{\cosheaf{F}[\covU_1] \ar@/^2pc/[rr] \ar[r] & \cosheaf{F}[\covU_2] \ar[r] & \cosheaf{F}(U)}
\]

The flaw in this reasoning is obvious because it uses the fallacious ``one out of three'' rule for isomorphisms. 
The correct ``two out of three'' rule for isomorphisms says that whenever there is a commutative triangle in a category where two of the arrows are isomorphisms, then the other arrow must be an isomorphism as well.

\begin{figure}
\includegraphics[width=.7\textwidth]{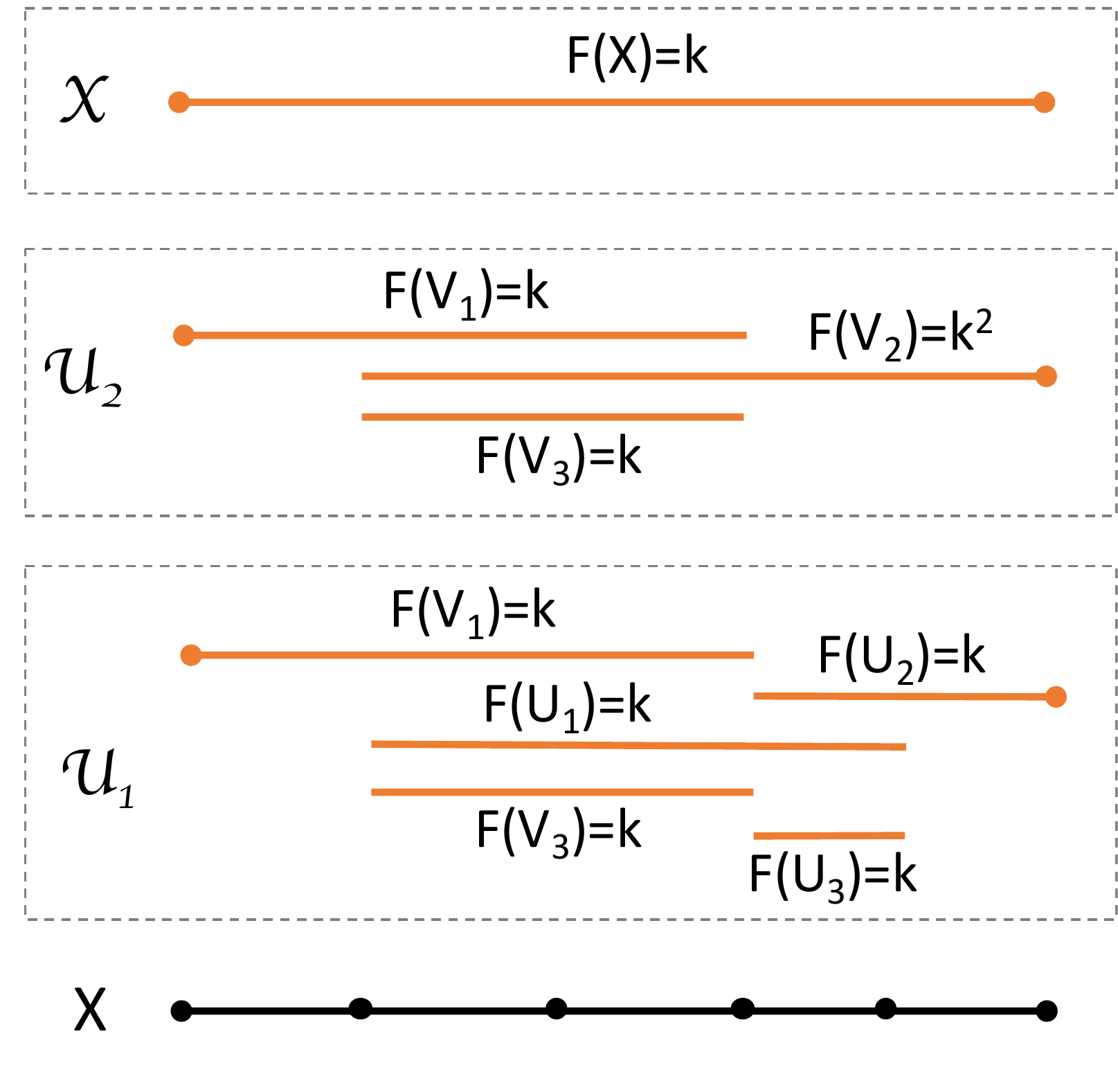}
\caption{A counter example to the statement that the cosheaf property is inherited by coarser covers.}
\label{fig:refinement-counter-example}
\end{figure}

Just to be clear that there is no way of repairing the above statement, consider the example drawn in Figure~\ref{fig:refinement-counter-example}.
Here $X$ can be viewed as the interval, but with a cell structure given by the indicated vertices.
For the covers considered in our example, we need only use unions of open stars of cells.
Near each open set we have indicated the value of a precosheaf $F$, valued in $\Bbbk$-vector spaces.
For our example, we have $F(X)=\Bbbk$, $F(V_2)=\Bbbk\oplus\Bbbk\cong \Bbbk^2$, and all other indicated open sets are also assigned the value $\Bbbk$.
For any open set $W$ contained in $V_2$, we declare the map $F(W)\to F(V_2)$ to be the inclusion into the first factor. To ensure functoriality, we declare the map $F(V_2)\to F(X)$ to be projection onto the first factor.
Clearly $\covU_1$ refines $\covU_2$, and even though the colimit of $F$ over $\covU_1$ is $\Bbbk$, the colimit of $F$ over $\covU_2$ is $\Bbbk^2$.
Indeed, this confirms the correct observation that the curved isomorphism decomposes
\[
\xymatrix{\cosheaf{F}[\covU_1] \ar@/^2pc/[rr]^{\cong} \ar@{^{(}->}[r] & \cosheaf{F}[\covU_2] \ar@{->>}[r] & \cosheaf{F}(X)}
\]
as an injection following by a surjection.








\section{Kan Extensions Define Cosheaves On Posets}

Let $\pP$ be a poset.
A \define{down-set} in $\pP$ is a subset $S\subseteq \pP$ with the property that whenever $q\in S$ and $p\leq q$, then $p\in S$.
It is easy to see that the collection of down-sets $\Down(\pP)$ defines a topology on $\pP$.
This is often called the \define{Alexandrov topology} and we can regard $\Down(\pP)=\Open(\pP)$ where down-sets are open sets.

The main theorem of this note, \cref{thm:kan-cosheaves}, states that whenever we have a functor $F:\pP \to \cat$, where $\pP$ is a poset and $\cat$ is a co-complete category, then the only step necessary to define a cosheaf is the left Kan extension.
To provide as self-contained a treatment as possible, we now review the key notions required to carefully define the left Kan extension as well as the notions of cofinality required to prove \cref{thm:kan-cosheaves}.

\begin{defn}\label{defn:comma-cat}
Suppose $E:\aat\to\bat$ is a functor and let $b$ be an object of $\bat$.
The \define{comma category under $b$}, written $(E\downarrow b)$, is defined as follows:
\begin{itemize}
	\item The objects of $(E\downarrow b)$ are morphisms in $\bat$ of the form $\alpha: E(a) \to b$  where $a$ is any object of $\aat$.
	\item A morphism of $(E\downarrow b)$ between two objects $\alpha: E(a) \to b$ and $\alpha': E(a') \to b$ is a morphism $\gamma: a \to a'$ in $\aat$ making the following diagram commute:
	\[
		\xymatrix{E(a) \ar[rr]^{E(\gamma)} \ar[rd]_{\alpha} & & E(a') \ar[ld]^{\alpha'} \\ & b &}
	\]
\end{itemize}
There is also a \define{comma category over $b$}, written $(b \downarrow E)$, that is defined completely dually: objects are morphisms in $\bat$ of the form $\alpha: b \to E(a)$ for some $a$ in $\aat$, morphisms are morphisms in $\aat$ making the dual triangle commute:
    \[
    	\xymatrix{ & b \ar[dl]_{\alpha} \ar[dr]^{\alpha'} &  \\ E(a) \ar[rr]_{E(\gamma)} & & E(a') }
    \]
\end{defn}

For an example of a comma category, we consider the special case of maps between posets.

\begin{ex}
Recall that a map of posets $f:\pP\to\pQ$ is equivalently a functor.
Substituting $f$ for $E$ in the above definition leads to the following interpretations:
The comma category $(f\downarrow q)$ is simply the sub-poset of $\pP$ consisting of those $p$ such that $f(p)\leq q$, which one might call the ``sublevel set of $f$ at $q$.''
The comma category $(d\downarrow f)$ is thus the superlevel set of $f$ at $q$.
\end{ex}

The comma category $(E\downarrow b)$ associated to a functor $E:\aat\to\bat$ and an object $b$ in $\bat$, has a natural projection functor $\pi^b : (E\downarrow b) \to \aat$ that sends an object $\alpha: E(a) \to b$ to the object $a$ in $\aat$, a morphism $\gamma:a\to a'$ goes to the same morphism in $\aat$.
This observation, and this particular choice of comma category, allows us to define the \emph{left Kan extension} of a functor $F:\aat \to \cat$ \emph{along} the functor $E:\aat \to \bat$.

\begin{defn}[Pointwise Kan Extensions, cf.~\cite{riehl2017category} Thm. 6.2.1]

The \define{left Kan extension} of $F:\aat \to \cat$ along $E:\aat \to \bat$ is a functor $\Lan_E F:\bat \to \cat$ that assigns to an object $b$ of $\bat$ the following colimit
\[
	\Lan_E F (b) = \varinjlim \left( (E \downarrow b) \xrightarrow[]{\pi} \aat \xrightarrow[]{F} \cat \right) = \varinjlim_{E(a) \rightarrow b} F(a)
\]
Morphisms are sent to corresponding universal maps between colimits. 
\end{defn}

\begin{ex}
Suppose $j: \pP \hookrightarrow \pQ$ is an inclusion of posets and suppose $F:\pP \to \cat$ is a functor.
The left Kan extension of $F$ along $j$ assigns to an element $q\in Q$ the colimit of $F$ over the sublevel set of $j$ at $q$.
Note that this uses the fact that there there is at most one morphism of the form $j(p) \leq q$.
Moreover, if the inclusion is full, i.e.~if $p\leq_{\pP} p'$ if and only if $j(p)\leq_{\pQ} j(p')$, then the colimit can be viewed as occurring over all $p\in \pP$ such that $j(p)\leq q$.
\end{ex}

The following example is of utmost importance.

\begin{ex}
Let $\iota: \pP \hookrightarrow \Down(\pP)$ denote the map of posets that sends $p\in \pP$ to the principal down-set $D_p$.
Let $S\in \Down(\pP)$ be an arbitrary down-set.
The reader is asked to convince themselves that the comma category $(\iota \downarrow S)$ is given by the full subcategory of $\pP$ whose objects are those $p\in S$, which we write as $\pP_S$.
Consequently, if we wish to consider the left Kan extension of a functor
$F:\pP \to \cat$, 
then we have that
\[
	\widehat{F}(S) := \Lan_{\iota} F(S) = \varinjlim \left( \pP_S \hookrightarrow \pP \to \cat \right) = \varinjlim_{p\in S} F(p).
\]
\end{ex}

Finally, we mention that comma categories are used in the definition of cofinality, which will be used in the proof of \cref{thm:kan-cosheaves}.

\begin{defn}\label{defn:cofinal}
A functor $E: \aat \to \bat$ is \define{cofinal} if for every object $b$ in $\bat$ the comma category $(b \downarrow E)$ is 
\begin{itemize}
\item non-empty, and 
\item connected.
\end{itemize}
Equivalently, a functor $E$ is cofinal if for every functor $F:\bat \to \cat$ to any category $\cat$ the induced map on colimits
\[
	\varinjlim F\circ E \to \varinjlim F
\] 
is an isomorphism.
\end{defn} 

\begin{rmk}
Note that the equivalence of these two definitions says that whether a diagram $F$ indexed by $\bat$ has the same colimit when restricted along $E:\aat \to \bat$ is dictated by the ``topological'' properties (nonemptiness and connectedness) of the comma categories $(b \downarrow E)$ for all objects $b$ in $\bat$.
Viewing these comma categories as fibers, the equivalence of the above two cofinality conditions is perhaps best viewed as a categorical analogue of the Vietoris Mapping Theorem.
\end{rmk}

We now prove the main theorem of this section. The following proof should replace the mistaken proof of Theorem 4.2.10 of~\cite{Curry2014}.

\begin{thm}
\label{thm:kan-cosheaves}

Let $F:\pP \to \cat$ be a functor from a poset $\pP$ to a co-complete category $\cat$.
Let $\iota: \pP \to \Down(\pP)$ denote the association of an element $p\in \pP$ to the principal down-set $D_p$.
The left Kan extension of $F$ along $\iota$, written $\widehat{F}$ below, is a basic cosheaf.
\[
\xymatrix{ \pP \ar[r]^{F} \ar[d]_{\iota} & \cat \\ \Down(\pP) \ar@{.>}[ru]_{\Lan_{\iota} F =: \widehat{F}} & }
\]

\end{thm}
\begin{proof}
Let $S\in \Down(\pP)$ be an arbitrary down-set in $\pP$ and let $\cV=\{V_i\}$ be a basic cover of $S$ by down-sets.
We must show that the induced map
\begin{equation}\label{eqn:cosheaf-check}
	\varinjlim_{V_i \in \cV} \widehat{F}(V_i) \to \widehat{F}(S) = \varinjlim_{p \in S} F(p)
\end{equation}
is an isomorphism.

The left hand side of the above equation is properly viewed as an iterated colimit, which suggests that the indexing category is a product category.
Instead of working with a product category, we will introduce an auxiliary category associated to the cover $\cV$, which we call $\calJ$.
The objects of $\calJ$ are pairs
\[
	(V_i,p) \qquad \text{where} \qquad V_i\in \cV \qquad \text{and} \qquad p\in V_i.
\]
There is a unique morphism from $(V_i,p) \to (V_j,q)$ if $V_i\subseteq V_j$ and $p\leq q$.
Notice that we have two natural projection functors:
\[
	\pi_1 :\cJ \to \cV \qquad \pi_2:\cJ \to \pP_S \qquad \text{where} \qquad \pi_1(V_i,p)=V_i \qquad \pi_2(V_i,p)=p
\] 
These fit into the following commutative diagram of categories
\[
	\xymatrix{\cJ \ar[r]^{\pi_1} \ar[d]_{\pi_2} & \cV \ar[d]_{\slfp} \\ \pP_S \ar[r]_{\slfp} & \star }
\]
where $\star$ is the category with only one object and one morphism.
The strategy of the proof is to show that isomorphism desired in Equation~\ref{eqn:cosheaf-check} can be demonstrated using commutativity of the above square and Kan extensions.
Note that $F:\pP \to \cat$ defines, by abuse of notation, a functor $F:\cJ \to \cat$ by assigning to each pair $(V_j,q)$ the value $F(q)$.
In other words $F:\cJ \to \cat$ is really defined by pulling back $F:\pP \to \cat$ along $\pi_2: \cJ \to \pP_S \subseteq \pP$.

The first major step in our proof is to rephrase $\widehat{F}(V_i)$ in terms of a left Kan extension along $\pi_1$. Specifically, we prove that for any $V_i\in \cV$ we have the following isomorphism that is natural in $V_i$:
\[
	\widehat{F}(V_i):= \varinjlim_{p\in \pP_{V_i}} F(p) \to \varinjlim_{(V_j,q) \mid V_j\subseteq V_i} F(V_j) =: \Lan_{\pi_1} F(V_i)
\]
To prove this, we show that the comma category $(\pi_1 \downarrow V_i)$ contains the cofinal system $\pP_{V_i}$.
Recall that the objects of $(\pi_1 \downarrow V_i)$ are arrows in $\cV$ of the form $\pi_1(V_j,q)=V_j \to V_i$. In other words, objects of $(\pi_1 \downarrow V_i)$ are in bijection with objects $(V_j,q)\in \cJ$ which satisfy the nested sequence of inclusions $q\in V_j \subseteq V_i$.
We will use the notation $(q \in V_j \subseteq V_i)$ to refer to an object of $(\pi_1 \downarrow V_i)$. 
We must check that the functor
\[
	\slfj_{V_i}: \pP_{V_i} \to (\pi_1 \downarrow V) \qquad \text{where} \qquad q \mapsto (q\in V_i \subseteq V_i) =: \slfj_{V_i}(q)
\]
satisfies the two requirements of cofinality given in Definition~\ref{defn:cofinal}. 
First, the fiber over any object $(q \in V_j \subseteq V_i)$ in $(\pi_1 \downarrow V_i)$ is non-empty because we always have the arrow
\[
	(q\in V_j\subseteq V_i) \to (q \in V_i \subseteq V_i)=\slfj_{V_i}(q).
\]
Moreover, the fiber is connected because whenever we have two objects in the fiber over $(q\in V_j\subseteq V_i)$
    \[
    	\xymatrix{ & (q\in V_j \subseteq V_i) \ar[dl] \ar[dr] &  \\ \slfj_{V_i}(p) & & \slfj_{V_i}(r) }
    \]
    there is always the third object $\slfj_{V_j}(q)$ connecting the other two.
      \[
    	\xymatrix{ & (q\in V_j \subseteq V_i) \ar[dl] \ar[d] \ar[dr] &  \\ \slfj_{V_i}(p) & \slfj_{V_i}(q) \ar[l] \ar[r] & \slfj_{V_i}(r) }
    \]  
This proves our first isomorphism, that $\widehat{F}(V_i)$ can be computed as a Kan extension along $\pi_1$.

The utility of this first isomorphism is that we can compute the colimit of $\widehat{F}$ over the cover $\cV$ as an iterated Kan extension, i.e.~since colimits are equivalent to a left Kan extension along the constant map $\slfp$, we have that 
\[
	\varinjlim_{V_i\in \cV} \widehat{F}(V_i) = \varinjlim_{V_i\in \cV} \varinjlim_{p\in \pP_{V_i}} F(p) \cong \mathsf{Lan}_{\slfp} \Lan_{\pi_1} F.
\]
Since the composition of Kan extensions is naturally isomorphic to the Kan extension of the composition we have
\[
	\Lan_{\slfp} \Lan_{\pi_1} F \cong \Lan_{\slfp \circ \pi_1} F = \Lan_{\slfp\circ \pi_2} F \cong \Lan_{\slfp}\Lan_{\pi_2} F. 
\]

The next part of the proof is to show that we have the following isomorphism, natural in $p\in \pP$:
\[
	\Lan_{\pi_2} F (p) \cong F(p)
\]
We do this by direct calculation.
Observe that $(\pi_2 \downarrow p)$ has objects that are pairs $(V_j,q)$ with $q\leq p$.
Now consider two objects in this comma category $(V_j,q)$ and $(V_k,r)$.
By necessity, $q\leq p$ and $r\leq p$.
Since $\cV$ is a cover of $S$, we have that there must be some $(V_i,p)$ with $V_i\in \cV$ and $p\in V_i$.
Moreover, since $\cV$ is a cover by \emph{down-sets}, we have that the intersections $V_j\cap V_i$ and $V_i \cap V_k$ are both non-empty, containing the elements $q$ and $r$, respectively.
Since $\cV$ is a \emph{basic cover}, we have the existence of other cover elements $V_{ij,q}$ and $V_{ik,r}$ containing $q$ and $r$ and contained in the intersections $V_i\cap V_j$ and $V_i\cap V_k$.
This implies we have the following diagram
\[
\xymatrix{ & \pi_2(V_{ij,q},q) \ar[ld]_{\id} \ar[rd] & & \pi_2(V_{ik,r},r) \ar[ld] \ar[rd]^{\id} & \\ \pi_2(V_j,q) & & \pi_2(V_i,p) & & \pi_2(V_j,r)}
\]
and hence any co-cone on this diagram (after applying the functor $F$) factors through the value $F(p)$, thereby implying that the colimit of $F$ over the comma category $(\pi_2 \downarrow p)$ is $F(p)$.
We note that this argument is not dependent on the cover element $V_i$ chosen to contain $p$.
Indeed if $V_{i'}$ were another cover element containing $p$, then we could intersect $V_i$ with $V_{i'}$ and find another cover element contained in the intersection and containing $p$.
Each of these elements and their inclusions of $\cJ$ is carried by $\pi_2$ to the constant diagram with value $p$.

\end{proof}

\section{Acknowledgements}

The author would like to thank Professor Marco Varisco at the University at Albany, State University of New York, for help structuring and arriving at the new proof of \cref{thm:kan-cosheaves}. 
Any remaining mistakes are the author's own.
Additionally, the author would like to thank the NSF for supporting their work under grant CCF-1850052.

\printbibliography

\end{document}